\documentclass[reqno,12pt,letterpaper]{amsart}
\usepackage[colorlinks=true,linkcolor=Red,citecolor=Green]{hyperref}
\usepackage{amsmath}
\usepackage{mathtools}
\usepackage{amsfonts}
\usepackage{mathrsfs}
\usepackage{amsthm}
\usepackage{fullpage}
 \usepackage{bbm}
\usepackage[dvipsnames]{xcolor}
\usepackage{amssymb}
\usepackage{hyperref}
 \def\smallsection#1{\smallskip\noindent\textbf{#1}.}
\title[Existence of a Non-Equilibrium Steady State for the 1d non-linear BGK equation on an interval]{Existence of a Non-Equilibrium Steady State for the non-linear BGK equation on an interval}

\author{Josephine Evans}
\email{Josephine.Evans@warwick.ac.uk}
\address{Warwick Mathematics Institute, Zeeman Building, University of Warwick, CV4 7AL}
\author{Angeliki Menegaki}
\email{angeliki.menegaki@dpmms.cam.ac.uk}
\address{DPMMS, University of Cambridge, Wilberforce Rd, Cambridge CB3 0WA, UK}


\newtheorem{thm}{Theorem}
\newtheorem{defn}{Definition}
\newtheorem{lemma}{Lemma}
\newtheorem{prop}{Proposition}
\newtheorem{remark}{Remark}

\newtheorem{condition}{Condition}

\begin{document}
\maketitle
\begin{abstract}
We show existence of a non-equilibrium steady state for the one-dimensional, non-linear BGK model on an interval with diffusive boundary conditions. These boundary conditions represent the coupling of the system with two heat reservoirs at different temperatures. The result holds when the boundary temperatures at the two ends are away from the equilibrium case, as our analysis is not perturbative around the equilibrium.  We employ a fixed point argument to reduce the study of the model with non-linear collisional interactions to the linear BGK. 
\end{abstract}
\tableofcontents
 
\section{Introduction}

This work is a contribution to the study of non-equilibrium steady states for non-linear kinetic equations. We study the existence of non-equilibrium steady states for the \textit{non-linear BGK equation} on bounded domains with diffusive boundary conditions. In this paper we look at the 1d case where the velocity variable is in $\mathbb{R}$ and the $x$ variable is in an interval with boundary conditions at different temperatures at each side. We show the existence of a non-equilibrium steady state and explore its properties. 

The BGK model of the Boltzmann's equation is a simple kinetic relaxation model introduced by  Bhatnagar, Gross and Krook in \cite{BGK54} as a toy model for Boltzmann flows. The evolution problem for the BGK model was first studied in \cite{Per89} and later in \cite{GP89} where global existence was proved and in \cite{PP93} where existence and uniqueness was proved for the initial-value problem in bounded domains. 

Here  we are interested in non-equilibrium phenomena, that is to say equations with steady states which are not Gibb's states and are induced by effects external to the system of study. In our case these external effects are present as diffusive boundary conditions. We show results which are not derived by perturbations models which have equilibrium states or by models which are close to the hydrodynamic regime. That is to say we work in the regime where the Knudsen number is not considered to be small.

We describe our model in the following subsection.

\subsection{Description of the model} 

We consider a gas of particles in the domain $(0,1)$ where the collisions among the particles are described by the nonlinear BGK operator. The distribution function $f(t,x,v)$ of the gas is the density of the particles at the position $x \in (0,1)$ with velocity $v \in \mathbb{R}$ at time $t>0$. We denote by $\kappa$ the \textit{Knudsen number}\footnote{The Knudsen number $\kappa$ is defined as the ratio between the mean free path and the typical observation length.} and
we study the existence of stationary solutions $f(x,v)$ to the following equation
\begin{align} \label{eq:nl}
&\partial_t f + v \partial_x f = \frac{1}{\kappa} \left(\rho_f \mathcal{M}_{T_f} - f\right), \\
&f(0,v) = \widetilde{\mathcal{M}}_1(v) \int_{v'<0} |v'|f(0,v')\mathrm{d}v', \qquad v>0,  \label{BC1}\\
&f(1,v) = \widetilde{\mathcal{M}}_2(v) \int_{v'>0} |v'| f(1,v') \mathrm{d}v', \qquad v<0. \label{BC2}
\end{align}
Here the \textit{spatial density} $\rho_f(x)$ and the \textit{pressure} $P_f(x):= \rho_f(x) T_f(x) $ are given respectively by 
\begin{align}
\rho_f(x) = \int_{-\infty}^\infty f(x,v) \mathrm{d}v, \qquad \rho_f(x) T_f(x) = \int_{-\infty}^\infty v^2 f(x,v) \mathrm{d}v 
\end{align} and then the \textit{local temperature} corresponding to $f$ is $T_f$. 
We denote by $\mathcal{M}_{T_f}(v)$ the Maxwellian with temperature $T_f$ i.e.
\[ \mathcal{M}_{T_f}(v) = (2\pi T_f)^{-1/2} \exp \left(-\frac{1}{2T_f}v^2 \right). \] Furthermore, $\widetilde{\mathcal{M}}_1, \widetilde{\mathcal{M}}_2$ are Maxwellians at the boundary temperatures $T_1, T_2$ respectively and they are considered to be renormalised so that, for $i=1,2$,
\begin{align} \label{eq: normalis}
\int_0^\infty v \widetilde{\mathcal{M}}_i (v) \mathrm{d}v = 1. 
\end{align} 
 This means that
\[ \widetilde{\mathcal{M}}_i (v) := \frac{1}{T_i} \exp \left( - \frac{1}{2T_i}v^2\right). \]
In other words we study the steady states of a gas which is coupled to two temperature reservoirs at the two boundaries of the domain $(0,1)$ and this coupling is implemented through the so-called \textit{diffusive boundary conditions} or \textit{Maxwell boundary conditions}. So that when particles hit one of the boundaries $\{0,1\}$, they get reflected and re-enter the domain with  new velocities drawn from the Maxwelians $\widetilde{\mathcal{M}}_i (v)$ corresponding to  temperatures $T_1, T_2$ at the two ends.

\subsection{State of the Art}
This paper is motivated by \cite{CELMM18, CELMM19}, where they study the non-linear BGK equation on the periodic torus in the presence of scatterers at two different temperatures. There it is straightforward to find one steady state and these papers show that this state is unique under certain conditions. There the non-equilibrium forcing is the same throughout the space. Also for the non-linear BGK equation there is a paper of Ukai \cite{U92}, about the existence of steady states with prescribed boundary conditions which is a situation similar to that studied here. The boundary conditions we consider here are different, since the paper of Ukai prescribes the density at either side of the interval, whereas in our paper we prescribe diffusive boundary conditions. The techniques are also different.

The majority of the papers investigating non-equilibrium steady states of kinetic equations are in the setting of the Boltzmann equation. We mention first the paper \cite{AN00} which also deals with a non-perturbative setting to show the existence of non-equilibrium steady states to the Boltzmann equation in the slab. This paper crucially uses the entropy distribution of the equation. There are also a number papers about similar problems in a perturbative setting either when the difference in temperatures is small \cite{EGKM13}, or when the Knudsen number is small and a hydrodynamic approximaiton can be used \cite{A00, AEMN10, AEMN12}. There have also been works in a spatially homogeneous setting in the presence of scatterers \cite{CLM15} for the Boltzmann equation and \cite{E16} for Kac's toy model for the Boltzmann equation. We also mention the preprint \cite{B19} which shows exponential convergence towards non-equilibrium steady states for the free transport equation in a domain with Maxwell boundary conditions.

One of the main motivations to study non-equilibrium phenomena, like the one appearing in the model of this note due to non-isothermal boundaries, is the better understanding of the \textit{Fourier's law}, which from a mathematical point of view is a very challenging problem. The Fourier's law, which is well tested for several materials, relates the macroscopic thermal flux $J(t,x)$ to the small variations of the gradient of the temperature $\nabla T(t,x)$: \begin{align} \label{eq: Fourier law}
J(t,x)= - K(T) \nabla T(t,x)
\end{align}
where $0 < K(T) < \infty$ is the thermal conductivity of the material.   It is not diffucult to see that \eqref{eq: Fourier law} implies the following diffusion equation for the temperature: $$ c(T) \partial_t T = \nabla(K(T) \nabla T) $$ where $c(T)$ is the specific heat of the system per unit volume. 

Concerning heat conduction in gases:  \eqref{eq: Fourier law} was rigorously proven in \cite{ELM94, ELM95} 
for the stationary Boltzmann equation in a slab for small Knudsen numbers and when the temperature difference is small. 
We also refer here again to  \cite{EGKM13} where the authors construct solutions to the $3$d steady problem with the Boltzmann hard spheres collision operator and diffusive boundaries with different temperatures at the two walls that do not oscillate too much. There the authors work with small temperature difference and they can see mathematically that the Fourier's law does not hold, since they are in the kinetic regime, by combining their result with pre existing numerical simulations in \cite{OAS89}. 

The abovementioned works are specific answers to the more general question in Statistical Physics: the mathematically rigorous derivation of Fourier's law or a proof of its breakdown, from microscopic, purely deterministic or stochastic,  models. For several overviews on this topic we refer the reader to \cite{BLR00, Lep16, Dhar08, BF19}. 

An example of such a microscopic heat conduction model, a model of heat reservoirs is the so-called chain of oscillators. For more information on this model we refer to \cite{RLL67, EPR99b, RBT02, Car07, CEHRB18} where questions of existence and uniqueness of a NESS and exponential approach towards it are addressed, as well as to \cite{Hair09, HM09} for interesting features of the model: for example cases where there is no spectral gap of the generator of the associated process. Quantitative results concerning the scaling of the spectral gap  in terms of the number of the particles, for special cases of the chain, can be found in \cite{Me19, BM20}.

\section{Mathematical Preliminaries}
First note that the normalisation \eqref{eq: normalis} is chosen so that the equation conserves mass, indeed we record this observation in the following Lemma: 
\begin{lemma}
The equation \eqref{eq:nl} at least formally conserves mass.  
\end{lemma}
\begin{proof} We write
\begin{align*} \frac{\mathrm{d}}{\mathrm{d}t} \int f(x,v) \mathrm{d}x\mathrm{d}v =& -\int v\partial_x f \mathrm{d}x \mathrm{d}v + \frac{1}{\kappa} \left(\int \rho(x) \mathcal{M}_T(v) \mathrm{d}v \mathrm{d}x -\int f\right)  \\
 =& -\int v \partial_x f = \int vf(1,v)\mathrm{d}v - \int v f(0,v) \mathrm{d}v . \end{align*}
Then we show that each of the boundary terms is zero:
\begin{align*}\int v f(1,v) \mathrm{d}v =& \int_{-\infty}^0 vf(1,v)\mathrm{d}v + \int _0^\infty v f(1,v) \mathrm{d}v \\
 =& \int_0^\infty |v'| f(1,v')\mathrm{d}v' \int_{-\infty}^0 \tilde{M}_2(v) v \mathrm{d}v + \int_0^\infty v f(1,v) \mathrm{d}v \\
=& 0. 
\end{align*} Similarly we show that the other boundary term is $0$ as well, which concludes the proof. 
\end{proof}

\subsection{Notation} We write $f(x) = \mathcal{O}(g(x))$ to denote that there is a constant $C>0$ such that $\left\lvert f(x) \right\rvert \le C \left\lvert g(x) \right\rvert$. We occasionally write $ A \lesssim B$  in order to say that $A \leq C B$ for some constant $C$ that only depends on the two temperatures $T_1,T_2$.  We denote by $\mathcal{B}(A)$ the Borel $\sigma$-algebra on $A$ and by $n_x$ the outward unit vector at $x \in \{0,1\}$.  We also write $C_c^{\infty}(X)$ for the space of the compactly supported smooth functions on $X$.

\subsection{Plan of the paper}
We introduce the main results in the next Section. In Section  \hyperref[subsection: fixed point thrm]{4} we present the proofs of the main results. This is split into subsections showing the different criteria needed to be fulfilled in order to apply the Schauder fixed point Theorem. In particular in subsection  \hyperref[subsection Linfty bounds]{4.3} we present the asymptotic behaviour of the moments in order to get an $L^{\infty}$ bound on the temperature profile, and then to prove H\"older continuity of it in \hyperref[subsection: Holder continuity]{4.4}.  In \hyperref[subsection: contin of F]{4.5} we prove the continuity of the mapping to which we apply Schauder's Theorem in   \hyperref[subsection: fixed point thrm]{4.6}.
 Finally, we conclude with Section \hyperref[sec: discussion]{5} with a discussion of the results and possible avenues for future research.

\section{Main Results}

We introduce the following condition, under which our main Theorem holds.
\begin{condition} \label{condition}
We say that the pair of temperatures $T_1,T_2$ satisfy this condition when
\begin{itemize} 
\item[(C1)] $\kappa^2 T_1 > \gamma_2$  and  \label{con1}
\item[(C2)]  $\sqrt{T_2} - \sqrt{T_1} \ge \gamma_1\kappa^{1/2} T_2^{1/4}$ \label{con2}
\end{itemize} 
 where $ \gamma_1, \gamma_2$ are positive constants and $\kappa >0$ is the time renormalizing constant in front of the collision operator in equation \eqref{eq:nl}, \textit{i.e.} the Knudsen number. 
\end{condition}

%
Our main results on the steady state of the nonlinear BGK operator with diffusive boundary conditions are summarized in the following Theorem. 
\begin{thm} \label{thm:main}
For every two fixed temperatures $T_1,T_2$ satisfying Condition \ref{condition}, there exists  a steady state which satisfies equation \eqref{eq:nl} with boundary conditions \eqref{BC1} and \eqref{BC2}.
Furthermore, this steady state has the following properties:
\begin{itemize}
\item It has zero momentum uniformly in  $x \in (0,1)$.
\item It has constant density and pressure equal to $\sqrt{T_1T_2}$, asymptotically with $T_1$. In particular, for all $x \in (0,1),$ \begin{align*}
1- \gamma_0 \kappa^{-1/2}T_1^{-1/4} \leq &\rho_f(x) \leq 1+ \gamma_1 \kappa^{-1/2}T_1^{-1/4} \\
 \sqrt{T_1T_2} \lesssim &P_f(x) \lesssim   \sqrt{T_1T_2}.
\end{align*}   
\item Its temperature profile is $1/2$-H\"older continuous and also it is asymptotically equal to  $\sqrt{T_1T_2}$ with the deviation from $\sqrt{T_1T_2}$ decreasing as $T_1$ increases: for all $x \in (0,1)$,
\begin{align*} 
\sqrt{T_1T_2}(1-  \gamma_1\kappa^{-1/2}T_1^{-1/4}) \lesssim T_f(x) \lesssim  \sqrt{T_1T_2}(1 +\gamma_0 \kappa^{-1/2}T_1^{-1/4}), 
 \end{align*} for some constants $\gamma_0, \gamma_1$ and $\kappa$ the constant in front of the collisional operator in \eqref{eq:nl}. 
\end{itemize} 
 
\end{thm}


\begin{remark}
The fact that the steady state of this equation has zero momentum uniformly in $x$, implies it is also a solution to the similar time independent equation,
\[ v\partial_x f = \rho_f(x) \mathcal{M}_{u_f,T_f}(v) -f, \] with the boundary conditions \eqref{BC1} and \eqref{BC2}. Here,
\[ \mathcal{M}_{u_f, T_f}(v) = (2\pi T_f)^{-1/2} e^{-(v-u_f)^2/T_f}, \] and
\[ \rho_f u_f := \int_\mathbb{R} f(x,v)v \mathrm{d}v. \]
\end{remark}

\section{Proofs} 

\subsection{Strategy of Proof}
In order to prove the existence of a steady state when $T_1, T_2$ satisfy condition \ref{condition}, we perform a fixed point argument. We look at the following linear equation with the same spatially variable diffusive boundary conditions for given temperature profile $T(x)$:
\begin{align} 
&\partial_t f+ v\partial_x f = \rho(x) \mathcal{M}_{T(x)} - f, \label{eq:steady} \\
&f(0,v) = \widetilde{M}_1(v) \int_{v'<0} |v'|f(0,v')\mathrm{d}v', \qquad v>0, \label{bc:left}\\
&f(1,v) =  \widetilde{M}_2(v) \int_{v'>0} |v'| f(1,v') \mathrm{d}v', \qquad v<0. \label{bc:right}
\end{align}
where $\mathcal{M}_{T(x)} $ is the Maxwelian with temperature $T(x)$.  
\begin{remark}
This differs from equation \eqref{eq:nl} since we use a fixed temperature profile in the Maxwellian on the right hand side rather than the temperature profile coming from $f$.
\end{remark}

We prove the two following facts:
\begin{itemize}
\item The PDE \eqref{eq:steady}-\eqref{bc:left}-\eqref{bc:right} is the equation on the law of a stochastic process and this stochastic process has a unique equilibrium state. This equilibrium steady state has a temperature profile which we call $\tau(x)$.
\item If $T_1 \leq T(x) \leq T_2$ and $T_1, T_2$ satisfy condition \ref{condition} then we have that $T_1 \leq \tau(x) \leq T_2$ and $\tau(x)$ is $1/2$-H\"older continuous with modulus of continuity depending on $T_1,T_2$.
\end{itemize}
We define the map $\mathscr{F}(T)=\tau$ which is a map between continuous functions on $(0,1)$ and thanks to the first fact above, it is well-defined. Then we apply the Schauder fixed point theorem using the second fact above to show that $\mathscr{F}$ has a fixed point.  From the definition of the mapping $\mathcal{F}$ we get that a fixed point implies that the temperature profiles of the nonlinear and the linear model will coincide. Therefore, for $T$ being this fixed point, corresponding to the two temperature profiles,  a steady state of the linear model \eqref{eq:steady}-\eqref{bc:left}-\eqref{bc:right},  will also be a steady state of the nonlinear model \eqref{eq:nl}-\eqref{BC1}-\eqref{BC2}. In the following sections we make it precise how to define the map $\mathscr{F}(T)$, then give bounds on $\mathscr{F}(T)$ which allow us to prove point 2. Finally, we use these to apply the Schauder fixed point theorem. 

\subsection{Definition of the map $\mathscr{F}(T)$.}
In this section we work in the case $\kappa =1$ in order not to track too many constants and to simplify the presentation, since quantitative bounds in this section do not have impact on our final result. In order to properly define the map $\mathscr{F}(T)$ we need a well defined way of selecting a steady state of the PDE \eqref{eq:steady}-\eqref{bc:left}-\eqref{bc:right}. In order to do this we define a stochastic process and show that this stochastic process has a unique steady state the law of which is a weak solution to the steady state version of \eqref{eq:steady}-\eqref{bc:left}-\eqref{bc:right}. First we define what we mean for a weak measure valued solution of \eqref{eq:steady}-\eqref{bc:left}-\eqref{bc:right}.
\begin{defn} \label{defin of weak sol}
A weak solution in the sense of measures to the PDE \eqref{eq:steady}-\eqref{bc:left}-\eqref{bc:right} is a triple $\mu_{1,t}, \mu_{2,t}, \mu_t$ with $\mu_i$ satisfying that for every test function supported on $\mathbb{R}_i = \mathbb{R}_+$ for $i=1$ or $\mathbb{R}_i = \mathbb{R}_-$ for $i=2$ 
\[ \int_{\mathbb{R}_i} \phi(v) \mu_i(\mathrm{d}v) = \int_{\mathbb{R}- \mathbb{R}_i} |v'| \mu_i(\mathrm{d}v') \int_{\mathbb{R}_i} v \widetilde{\mathcal{M}}_i(v) \phi(\xi,v) \mathrm{d}v \] where $\xi=0,1$, the left boundary for $i=1$ and the right boundary for $i=2$.   
Furthermore
\begin{align*} \int_0^\infty &\iint_{(0,1) \times \mathbb{R}} \left(\partial_t \phi(t,x,v) + v \partial_x \phi(t,x,v) + \int_{\mathbb{R}} \phi(t,x,v')\mathcal{M}_{T(x)}(v')\mathrm{d}v' - \phi(t,x,v) \right) \mu_t(\mathrm{d}x, \mathrm{d}v)\mathrm{d}t \\&+ \iint_{(0,1) \times \mathbb{R} } \phi(0,x,v) \mu_0(\mathrm{d}x \mathrm{d}v) = \int_0^\infty  \int_\mathbb{R} \phi(t,1,v) \mu_2(\mathrm{d}v) \mathrm{d}t - \int_0^\infty  \int_\mathbb{R}  \phi(t, 0,v)\mu_1(\mathrm{d}v) \mathrm{d}t .
\end{align*}
\end{defn}

\smallsection{Existence and uniqueness for the linear BGK equation with diffusive boundary conditions}
For our purposes we give a probabilistic interpretation of the evolution of the process for the linear BGK. Note also that  similar techniques for the free transport equation with diffusive and specular reflective boundary conditions for higher dimensions have been applied in \cite{BeFo19}. We work on the level of stochastic processes because there is the possibility of some non-uniqueness occuring at what is known as the `grazing set' which in our case is the two points $(0,0), (1,0)$. Defining a stochastic process allows us to set values at these points.

\begin{prop} \label{prop:welldefined}
For every given continuous function $T(x)$ there exists a well defined way in which we can select a triple of measures $\mu_1, \mu_2, \mu$ which is a stationary solution in a weak sense to the linear PDE \eqref{eq:steady}-\eqref{bc:left}-\eqref{bc:right}.
\end{prop}
We split the proof into two lemmas.
We first construct a stochastic process and show that it is well defined and then show that the law of this is a desired weak solution. 
\begin{defn}[Construction of the stochastic process] \label{stochastic process}
Let $(R^{1}_i)_{i \geq 1},(R^{2}_i)_{i \geq 1},  $  be two sequences of random variables with $R^{1}_i$ having law $v \widetilde{\mathcal{M}}_1$ and $R^{2}_i$ having law $|v| \widetilde{\mathcal{M}}_2$. Furthermore let $N_i$ be a set of $N(0,1)$ random variables and $S_i$ be a sequence of exponential random variables with rate $1$. Now we define the deterministic map
\[ \zeta(x,v) = \inf \{ s >0, \hspace{3pt} x+vs \in \{0,1\}\} \]  which is our first collision with one of the boundaries. Then we define recursively
\[ T_{k+1} = T_k + \min \{S_{k+1}, \zeta(X_{T_k}, V_{T_k})  \}. \] Then for $t \in [T_k, T_{k+1})$ we have
\[ X_t = X_{T_k} + (t-T_k)V_{T_k}, V_t = V_{T_k}. \] We jump at the times $T_k$ so that
\[ V_{T_{k+1}} = \mathbbm{1}_{T_{k+1}-T_k = S_{k+1}}\sqrt{T(X_{T_k})}N_{k+1} + \mathbbm{1}_{X_{T_{k+1}}=0}R^1_{k+1} + \mathbbm{1}_{X_{T_{k+1}}=1} R^2_{k+1} \]
where here $T(X_{T_k})$ is the temperature at the position $X_{T_k}$.
\end{defn}
\begin{lemma}[Non-explosion of the process]
This stochastic process defined in \ref{stochastic process} is well defined and exists for all $t>0$.
\end{lemma}
\begin{proof}
We would like to show that this process is non-explosive i.e. $T_i \rightarrow \infty$ almost surely. Lets look at the event 
\[  A_k = \{ R^1_{2k+1} < 1, R^1_{2k+2} <1, R^2_{2k+1} <1, R^2_{2k+2}, S_{2k+1}>1, S_{2k+2}>1  \}. \] The $A_k$'s are independent events and $\mathbb{P}(A_k)=\mathbb{P}(A_1) = p>0$. Therefore by Borel-Cantelli $A_k$ happens infinitely often almost surely. On $A_k$ we can see that $T_{2k+2}-T_{2k}>1$. This is because $A_k$ ensures that $T_{2k+1} -T_{2k}>1$ or $\zeta(X_{T_{2k}},V_{T_{2k}})<1$, and in the second case we know that $X_{T_{2k+1}} \in \{0,1\}$ so the next jump time is defined by $R^1_{2k+1}, R^2_{2k+2}$ and $S_{2k+2}$ which are all chosen so that $T_{2k+2}-T_{2k+1}>1$ if $X_{T_{2k+1}} \in \{0,1\}$. 
\end{proof}

\begin{lemma}
The law of this stochastic process is a weak solution to the SDE.
\end{lemma}
\begin{proof} 
Here we follow \cite{BeFo19}. We begin by showing how we can represent the boundary measures: for a set $A \in \mathcal{B}( (0,\infty) \times   \{0,1\}\times \Sigma_{\pm})$ with $\Sigma_{\pm}:=  \{ (x,v) \in \{0,1\} \times \mathbb{R}: \pm v \cdot n_x<0 \}$, we 
  introduce the measures
\begin{align*} \mu_{-}^i(A) &=  \mathbb{E}\left( \mathbbm{1}_{(T_i, X_{T_i}, V_{T_i}) \in A} \mathbbm{1}_{T_i = \zeta(X_{T_{i-1}},  V_{T_{i-1}})}\right),\\ \mu_+^i(A) &= \mathbb{E}\left( \mathbbm{1}_{(T_i, X_{T_i}, V_{T_i-}) \in A} \mathbbm{1}_{T_i = \zeta(X_{T_{i-1}}, V_{T_{i-1}})}\right)  
\end{align*}
\textit{i.e.} $\mu_-^i$ is the law of the triple $ (T_i, X_{T_i},V_{T_i})$, \textit{i.e.} after the collision with a boundary, and  $\mu_+^i$ is the law of the triple $ (T_i, X_{T_i},V_{T_i-})$, \textit{i.e.} exactly before the collision with a boundary.
 Then we have
 \begin{align*}  \mu_{+}(A)&= \sum_i \mu_{+}^i(A)\quad \text{for}\ A \in \mathcal{B}((0,\infty)\times \Sigma_-) ,\\
 \mu_{-}(A)&= \sum_i \mu_{-}^i(A)\quad \text{for}\ A \in \mathcal{B}((0,\infty)\times \Sigma_+) . 
 \end{align*}
These boundary measures satisfy the desired boundary condition. Indeed, we investigate the relationship between these two measures:
\begin{align*}
\mu_-^i(A) = & \mathbb{E}\left( \mathbbm{1}_{(T_i, X_{T_i}, V_{T_i}) \in A} \mathbbm{1}_{T_i = \zeta(X_{T_{i-1}}, V_{T_{i-1}})}\right)\\
=& \mathbb{E} \left(\mathbbm{1}_{X_{T_i}=0} \mathbbm{1}_{(T_i, X_{T_i}, R^1_i) \in A} + \mathbbm{1}_{X_{T_i}=1} \mathbbm{1}_{(T_i, X_{T_i}, R^2_i) \in A} \right)\\
=& \int \mathbb{E}\left( \mathbbm{1}_{X_{T_i}=0} \mathbbm{1}_{(T_i, X_{T_i}, v) \in A}\right) v \widetilde{\mathcal{M}}_1(v) \mathrm{d}v +\int \mathbb{E}\left( \mathbbm{1}_{X_{T_i}=1} \mathbbm{1}_{(T_i, X_{T_i}, v) \in A}\right) v \widetilde{\mathcal{M}}_2(v) \mathrm{d}v. 
\end{align*}
Therefore,
\begin{align*}
\mu_-^i(A) =& \iint_{(0,T)\times(0,\infty)}  \int_{-\infty}^0\left(\mathbbm{1}_{(t,0,w) \in A} \right)w \widetilde{\mathcal{M}}_1(w) \mathrm{d}w \mu_+^i(\mathrm{d}t,  \mathrm{d}v) \\
& + \iint_{(0,T)\times(0,\infty)} \int_0^\infty \left(  \mathbbm{1}_{(t,1,w) \in A} \right) w \widetilde{\mathcal{M}}_2(w) \mathrm{d}w \mu_+^i(\mathrm{d}t,  \mathrm{d}v).
\end{align*}
Testing against a test function $\phi \in C_c^{\infty} (\mathbb{R}_{\pm})$  we recover the boundary conditions as in the Definition \ref{defin of weak sol}. 
Now we would like to show that the tripple will be a weak solution to the PDE. For $$\mu_1 = \mathbbm{1}_{\{x=0\}}(\mu_+ + \mu_-),\quad \mu_2 = \mathbbm{1}_{\{x=1\}}(\mu_++\mu_-),$$  and $\phi \in C_c^{\infty}((0,\infty) \times (0,1) \times \mathbb{R})$,  we Taylor expand around $(t,X_t,V_t)$ and we write
\begin{align*} &\mathbb{E} \left( \phi(t+s, X_{t+s}, V_{t+s}) - \phi(t, X_t, V_t ) \right) =\\ &\qquad
 \mathbb{E}\left( \left(\phi(t+s, X_{t+s}, V_{t+s}) - \phi(t, X_t, V_t ) \right)  \mathbbm{1}_{\{X_t + sV_t  \in (0,1)\}}\right)+\\ & \qquad\qquad\qquad\qquad+  \mathbb{E}\left( \left(\phi(t+s, X_{t+s}, V_{t+s}) - \phi(t, X_t, V_t ) \right)  \mathbbm{1}_{\{X_t + sV_t  \notin (0,1)\}}\right) = \\
& \mathbb{E} \left( \left(s \partial_t \phi (t, X_t, V_t)+ s V_t \partial_x\phi(t,X_t,V_t)\right)  \mathbbm{1}_{\Big\{ \substack{\mbox{0 jumps in} \\ \mbox{$(t,t+s)$} }\Big\}}       \right)\\
&+ \mathbb{E} \left( \int_{-\infty}^\infty \left( \phi(t,X_t, w)- \phi(t, X_t, V_t) \right) (2\pi T(X_t))^{-1/2} \exp \left( - \frac{w^2}{2T(X_t)} \right) \mathrm{d}w  \mathbbm{1}_{\substack{\big\{\mbox{1 jump in} \\ \mbox{$(t,t+s)$} \big\}}}    \right) + \mathcal{O}(s) \\
&+ \mathbb{E} \left( \int_0^\infty \phi(t, X_t, w) w \widetilde{\mathcal{M}}_1(w) \mathrm{d}w \mathbbm{1}_{\{X_t + sV_t < 0\}}\mathbbm{1}_{\big\{\mbox{0 jumps in $(t,t+s)$}\big\}} \right) \\
&+ \mathbb{E} \left( \int_{-\infty}^0 \phi(t, X_t, w) w \widetilde{\mathcal{M}}_2(w) \mathrm{d}w \mathbbm{1}_{\{X_t +sV_t >1\}}\mathbbm{1}_{\big\{\mbox{0 jumps in $(t,t+s)$}\big\}} \right). 
\end{align*}
Letting $s$ to go to $0$, this gives us that \begin{align*}
&\int \left( \partial_t \phi(t,x,v) +v \partial_x \phi(t,x,v) \right) \mu_t(\mathrm{d}x, \mathrm{d}v) + \int \int_w ( \phi(t,x,w) - \phi(t,x,v)) \mathcal{M}_{T(x)}(w) dw \mu_t(\mathrm{d}x, \mathrm{d}v)\\
& + \int \int (\phi(t,0,w) - \phi(t,0,v)) w \widetilde{\mathcal{M}}_1(w) \mu_1( \mathrm{d}v)\\
& + \int \int (\phi(t,1,w) - \phi(t,1,v)) w \widetilde{\mathcal{M}}_2(w) \mu_2(\mathrm{d}v)=0.
\end{align*}
Therefore, $\mu_t$ is a weak solution to the PDE according to the Definition \ref{defin of weak sol}. 
\end{proof}

In order to prove the existence and uniqueness of a steady state for this stochastic process we use Doeblin's Theorem (we can find this in \cite{H10} for example). Which is as follows
\begin{condition}[Doeblin's condition] \label{DC}
If $\mathcal{P}$ is a stochastic semigroup acting on probability measures over a set $\Omega$ then $\mathcal{P}$ satisfies Doeblin's condition if there exists $\alpha \in (0,1)$ and $\nu \in \mathscr{P}(\Omega)$, a probability measure on $\Omega$ such that for every $z \in \Omega$ we have
\[ \mathcal{P}\delta_z \geq \alpha \nu.  \] 
\end{condition}
\begin{thm}[Doeblin's Theorem]
If $\mathcal{P}$ satisfies Doeblin's condition then it has a unique steady state.
\end{thm}

\begin{lemma} \label{lem: uniq of f for linear model}
Let $\mathcal{P}_t$ be the stochastic semigroup corresponding to the evolution of the stochastic process defined in \ref{stochastic process}, then there exist a time $t_*$ such that $\mathcal{P}_{t_*}$ satisfies Doeblin's condition \ref{DC}. In particular, the stochastic process has a unique steady state. 
\end{lemma}
\begin{proof}
We wish to find a lower bound for Doeblin's condition. We apply Duhamel's formula to find that have that 
\begin{align} \label{du:interior} x-vt \in (0,1), \hspace{10pt} f(t,x,v) = e^{-t} f\left(0,x-vt,v\right) + \int_0^t e^{-(t-s)} \rho(x-v(t-s)) \mathcal{M}_{T(x-v(t-s)}(v) \mathrm{d}s. \end{align} Similarly,
\begin{align} \label{du:leftright}  &x-vt \leq 0, \hspace{10pt} f(t,x,v) = e^{-x/v} f\left(t-\frac{x}{v},0,v\right) + \int_0^{x/v} e^{-(x/v-s)} \rho(vs)\mathcal{M}_{T(vs)}(v) \mathrm{d}s,\quad \text{and}\\
& x-vt \geq 1, \hspace{10pt} f(t,x,v) = e^{\frac{(1-x)}{|v|}} f\left(t-\frac{(1-x)}{|v|},1,v\right) + \int_0^{\frac{(1-x)}{|v|}} e^{-\left(\frac{-(1-x)}{|v|}-s\right)} \rho(1-vs) \mathcal{M}_{T(1-vs)}(v) \mathrm{d}s. 
\end{align} 
In light of this, we define
\[ R(t,x,v) := \left\{ \begin{array}{ll} x/v, &\text{for}\ x/v \leq t \\
t, &\text{for}\ x/v > t, v>0 \\
t, &\text{for}\ v=0 \\
t, &\text{for}\ (1-x)/|v| > t, v<0 \\
(1-x)/|v|, &\text{for}\ -(1-x)/v \leq t \end{array} \right. \] and 
\[ \Pi f (x,v) := \rho_f(x) \mathcal{M}_{T(x)}(v). \]
Then we have the following lower bound
\begin{align} \label{du:final} 
f(t,x,v) \geq & \int_0^R e^{-R} (\Pi f)(s, x-v(R-s),v) \mathrm{d}s.
\end{align}


Regarding the boundary conditions, we substitute in the first term from \eqref{du:interior}:
\[ f(t,0,v) = \widetilde{\mathcal{M}}_1(v) \int_{-\infty}^0 |u|f(t,0,u) \mathrm{d}u \geq \widetilde{\mathcal{M}}_1(v) \int_{-1/t}^0 e^{-t} |u|f(0,-ut,u) \mathrm{d}u \] and
\[ f(t,1,v) = \widetilde{\mathcal{M}}_2(v) \int_0^\infty |u| f(t,1,u)\mathrm{d}u \geq \widetilde{\mathcal{M}}_2(v) \int_0^{1/t} e^{-t} |u| f(0,1-ut,u) \mathrm{d}u. \] 
Now if we consider the initial condition $f(0,x,v) = \delta_{x_0}(x)\delta_{v_0}(v)$, 
we have 
\[ f(t,0,v) \geq e^{-t} \widetilde{\mathcal{M}}_1(v) |v_0| \delta_{x_0}(-v_0 t)  \mathbbm{1}_{t|v_0| \leq 1}\quad \text{and}\]
\[ f(t,1,v) \geq e^{-t} \widetilde{\mathcal{M}}_2(v) |v_0| \delta_{x_0}(1-v_0t) \mathbbm{1}_{t|v_0| \leq 1}. \]
Then we have
\[ x-vt \leq 0, \hspace{10pt} f(t,x,v) \geq e^{-t} \widetilde{\mathcal{M}}_1(v) |v_0| \delta_{x_0}\left(-v_0\left(t-\frac{x}{v}\right) \leq 1 \right). \] We need to know the local density and integrate in $v$. This gives us when $v_0 <0$,
\begin{align*}
\rho(t,x) \geq & \int_{x/t}^{x/(t-1/|v_0|)_+} e^{-t} \widetilde{\mathcal{M}}_1(v)|v_0| \delta_{x_0}\left(-v_0 \left(t-\frac{x}{v} \right)\right)\mathrm{d}v\\
\geq & \int_0^{\operatorname{min}(|v_0|t,1)} e^{-t} |v_0| \frac{1}{x|v_0|} \left( \frac{xv_0}{y+v_0t}\right)^2 \widetilde{\mathcal{M}}_1 \left( \frac{xv_0}{y+v_0t} \right) \delta_{x_0}(y) \mathrm{d}y \\
\geq & \mathbbm{1}_{x_0+v_0t \leq 0} \hspace{3pt} e^{-t} \frac{xv_0^2}{(x_0+v_0 t)^2} \widetilde{\mathcal{M}}_1 \left( \frac{xv_0}{x_0+v_0t} \right). 
\end{align*} 

Also when $v_0 >0$, we have
\begin{align*}
\rho(t,x) \geq & \int_{-\infty}^{(1-x)/|v|} e^{-t} \widetilde{\mathcal{M}}_2(v) |v_0| \delta_{x_0} \left( 1-v_0 \left( t-\frac{1-x}{|v|}\right)\right) \mathrm{d}v \\
\geq & e^{-t} |v_0| \frac{(1-x)v_0}{(x_0-1+v_0t)^2} \widetilde{\mathcal{M}}_2 \left( \frac{(1-x)v_0}{1-v_0t-x_0} \right)\mathbbm{1}_{x_0+v_0t \geq 1}.
\end{align*}

Now in the simplest case where $x_0+v_0t \in (0,1)$ we have
\[ \rho(t,x) = \delta_{x_0+v_0t}(x). \]

Now we are going to focus on the case where $x-vt \in (0,1)$ in which case
\[ f(t,x,v) \geq \int_0^t e^{-(t-s)} \rho(s, x-v(t-s)) \mathcal{M}_{T(x-v(t-s))}(v) \mathrm{d}s. \]
We have
\[ \mathcal{M}_{T(x)}(v) \geq \frac{1}{\sqrt{2\pi T_2}} e^{-v^2/2T_1} \geq \sqrt{\frac{T_1}{T_2}} \mathcal{M}_{T_1}(v) := \alpha G(v). \] Using this we can write
\[ f(t,x,v) \geq \alpha G(v) \int_0^t e^{-(t-s)} \rho(s,x-v(t-s)) \mathrm{d}s. \]

For a fixed $\epsilon>0$, we consider the following three cases 
\begin{enumerate}
\item $v_0<0, x_0/|v_0| \leq \epsilon$,
\item $v_0>0, (1-x_0)/v_0 \leq \epsilon,$
\item Neither of these holds.
\end{enumerate} 
We observe that in case (1) we know that $t \geq \epsilon$ implies that $x_0+v_0t \leq 0$ and in case (2): if $t \geq \epsilon$ then $x_0+v_0t \geq 1$.

\textit{For the case (1)}, we have
\[\rho(t,x) \geq e^{-t} \mathbbm{1}_{t \geq \epsilon} \frac{xv_0^2}{(x_0+v_0 t)^2} \widetilde{\mathcal{M}}_1 \left( \frac{xv_0}{x_0+v_0t} \right) \geq e^{-t} \mathbbm{1}_{t \geq \epsilon} \frac{x}{T_1t^2} e^{-1/2T_1\epsilon^2} . \]

Now we can substitute this into \eqref{du:final} again to get that
\begin{align*}
 f(t,x,v) &\geq \alpha e^{-t} G(v) \frac{1}{T_1 t^2} e^{-1/2T_1 \epsilon^2} \int_\epsilon^t (x-v(t-s)) ds \\ &= \frac{\alpha e^{-t}}{T_1t^2} G(v) e^{-1/2T_1 \epsilon^2} (t-\epsilon)\left(x-\frac{v}{2}(t-\epsilon) \right) \mathbbm{1}_{t\geq \epsilon}. 
 \end{align*}
  If we set $t_* = 2\epsilon$ then we have
\[ f(t_*,x,v) \geq \alpha e^{-2 \epsilon} G(v) \frac{1}{T_1 \epsilon^2} e^{-1/2T_1 \epsilon^2} \epsilon^2 \mathbbm{1}_{x-v\epsilon \in (\epsilon, 1-\epsilon)}. \]

\textit{For the case (2)}, we work essentially the same as in case (1) to get that
\[ f(t,x,v) \geq \frac{\alpha e^{-t}}{T_2 t^2} G(v) e^{-1/2T_2 \epsilon^2} (t-\epsilon)\left(1-x + \frac{v}{2}(t-\epsilon) \right). \] Setting $t_* = 2\epsilon$ we have
\[ f(t_*,x,v) \geq \alpha e^{-2 \epsilon} G(v) \frac{1}{T_2 \epsilon^2} e^{-1/2T_2 \epsilon^2} \epsilon^2 \mathbbm{1}_{x-v\epsilon \in (\epsilon, 1-\epsilon)}. \]

\textit{Finally for the third case}, we will need further iterations. Initially, we get that 
\[ \rho(t,x) \geq e^{-t} \delta_{x_0+tv_0}(x) \mathbbm{1}_{t \leq \epsilon}. \] We substitute this once into \eqref{du:final} to get
\[ f(t,x,v) \geq e^{-t} \alpha \int_0^t \delta_{x_0+sv_0}(x-v(t-s)) G(v) \mathbbm{1}_{t \leq \epsilon}\mathrm{d}s. \] Integrating in $v$ this, gives
\[ \rho(t,x) \geq e^{-t}\mathbbm{1}_{t\leq \epsilon} \alpha \int_0^t \int_{-\infty}^\infty
 \delta_{x_0+sv_0}(x-v(t-s)) G(v) \mathrm{d}v \mathrm{d}s. \] 
 After a change of variables, this is bounded below by
\begin{align*} 
\rho(t,x) \geq e^{-t}\mathbbm{1}_{t \leq \epsilon} \frac{\alpha}{t} \int_0^t G \left( \frac{x-x_0-v_0s}{t-s} \right)\mathrm{d}s \geq e^{-t}\mathbbm{1}_{t \leq \epsilon} \frac{\alpha}{t} \int_0^t G \left( \frac{1}{t-s} \right) \mathrm{d}s .  
\end{align*}
 We subsititute this back into \eqref{du:final} to get
\begin{align*} f(t,x,v) \geq & e^{-t} \frac{\alpha^2}{t} G(v) \int_0^t \mathbbm{1}_{s \leq \epsilon} \int_0^s G\left( \frac{1}{s-r}\right) \mathrm{d}r \mathrm{d}s \\
\geq & e^{-t} \frac{\alpha^2}{t} G(v) G(2/\epsilon) \int_{\epsilon}^{t} \int_0^\epsilon \mathrm{d}r \mathrm{d}s \\
\geq &  e^{-t} \frac{\alpha^2\epsilon (t-\epsilon)}{2t} G(v) G(2/\epsilon). \end{align*} Setting $t_*=2\epsilon$ we have,
\[ f(t_*,x,v) \geq e^{-2\epsilon}\frac{\alpha^2 \epsilon^2}{2\epsilon} G(v) G(2/\epsilon) \mathbbm{1}_{x-2v\epsilon \in (0,1)}. \]

Now let us set
\[ \beta = \alpha e^{-2\epsilon} \min \left\{\frac{\alpha \epsilon}{2} G\left( \frac{2}{\epsilon} \right), \frac{1}{T_1} e^{-1/2T_1 \epsilon^2} , \frac{1}{T_2} e^{-1/2T_2 \epsilon^2} \right\}. \] Then in every case we have that
\[ f(t_*,x,v) \geq \beta \mathbbm{1}_{x-2\epsilon v \in (\epsilon, 1-\epsilon)}. \]
\end{proof}

\subsection{$L^\infty$ Bounds on $\mathscr{F}(T)$.} \label{subsection Linfty bounds}
As we have uniqueness of a steady state for \eqref{eq:steady}-\eqref{bc:left}-\eqref{bc:right}, thanks to Lemma \ref{lem: uniq of f for linear model}, for $f$ being this solution with temperature profile $T$, we define the mapping  \begin{align*} \mathcal{F}: C((0,1)) \to C((0,1)),\quad  T\ \longmapsto\  \frac{\int f(x,v)|v|^2 \mathrm{d}v}{\int f(x,v) \mathrm{d}v}.
\end{align*}
In this subsection we first represent the solution to \eqref{eq:steady}- \eqref{bc:left}- \eqref{bc:right} in terms of the moments appearing in the boundary conditions. Then we check their asymptotic behaviour concerning the boundary temperatures $T_1,T_2$ so that we establish the conditions required on these temperatures in order to bound the $\mathcal{F}(T)(x)$ uniformly in $x$. The goal is to prove the proposition, 
\begin{prop} \label{prop:linftybounds}
If $T_1,T_2$ satisfy condition \ref{condition}, then we have that
\[ T_1 \leq \tau_T(x) \leq T_2\] uniformly in $x$. 
\end{prop}

We begin with the following Lemma.

\begin{lemma}
For $f$ a solution to \eqref{eq:steady}-\eqref{bc:left}-\eqref{bc:right} we have the following representation
\begin{align}
f(x,v) =& e^{-x/\kappa |v|}f(0,v) + \int_0^x e^{-(x-y)/\kappa |v|} \frac{1}{\kappa |v|} \rho(y) \mathcal{M}_{T(y)} \mathrm{d}y \qquad v>0,\\
f(x,v) =& e^{-(1-x)/\kappa |v|}f(1,v) + \int_x^1 e^{-(y-x)/\kappa |v|} \frac{1}{\kappa |v|} \rho(y) \mathcal{M}_{T(y)} \mathrm{d}y \qquad v<0.
\end{align}
\end{lemma}

\begin{proof}
We will use Duhamel's formula to get an exponential formulation for the equation:  let $v>0$

\[ \partial_t \left( e^{t/\kappa }f(vt,v)\right) =\frac{1}{\kappa} e^{t/\kappa } \rho(vt)\mathcal{M}_{T(vt)}(v).\]
Integrating this gives that
\[ e^{t/\kappa} f(vt,v) = f(0,v) + \frac{1}{\kappa}\int_0^t e^{s/\kappa} \rho(vs) \mathcal{M}_{T(vs)}(v) \mathrm{d}s. \] Now we write $x=vt$ and in the integral we make the change of variables $y=vs, \mathrm{d}y = v \mathrm{d}s$. This gives
\[ e^{x/\kappa v}f(x,v) = f(0,v) + \int_0^x e^{y/\kappa v} \frac{1}{\kappa v} \rho(y) \mathcal{M}_{T(y)}(v) \mathrm{d}y.  \]

Similarly, if $v<0$ we can write
\[ \partial_t \left( e^{t/\kappa} f(1+vt,v)\right) = \frac{1}{\kappa} e^{t/\kappa} \rho(1+vt) \mathcal{M}_{T(1+vt)}(v), \] again integrating this yields,
\[ e^{t/\kappa} f(1+vt,v) = f(1,v) +\frac{1}{\kappa} \int_0^t e^{s/\kappa} \rho(1+vs) \mathcal{M}_{T(1+vs)}(v) \mathrm{d}s. \] Now we make the change of variables $x=1+vt, y =1+vs$ this gives
\[ e^{(1-x)/\kappa|v|}f(x,v) = f(1,v) + \int_x^1 e^{(1-y)/\kappa |v|}\frac{1}{\kappa |v|} \rho(y) \mathcal{M}_{T(y)}(v) \mathrm{d}y. \]
\end{proof}

The above  lemma will give us a close form for the moments appearing in the boundary conditions. We start with the following definitions.
\begin{defn} \label{defn:C} We define the following moments
\begin{align}
C_- =& \frac{1}{\kappa}\int_0^1 \int_{v<0} e^{-y/\kappa |v|} \rho(y) \mathcal{M}_{T(y)}(v) \mathrm{d}v \mathrm{d}y,\\
C_+ =& \frac{1}{\kappa} \int_0^1 \int_{v>0} e^{-(1-y)/\kappa |v|} \rho(y) \mathcal{M}_{T(y)}(v) \mathrm{d}v \mathrm{d}y,\\
C_1=& \int_{v>0} |v| e^{-1/\kappa |v|} \widetilde{\mathcal{M}}_1(v) \mathrm{d}v <1,\\
C_2 =& \int_{v<0} |v| e^{-1/\kappa |v|} \widetilde{\mathcal{M}}_2(v)\mathrm{d}v <1.
\end{align}
\end{defn}

\begin{lemma}
The moments appearing in the boundary conditions can be written as
\begin{align*}
\int_{v<0} |v| f(0,v) \mathrm{d}v = \frac{1}{1-C_1C_2} \left( C_- + C_2 C_+ \right),\\
\int_{v>0} |v| f(1,v) \mathrm{d}v  = \frac{1}{1-C_1 C_2} \left( C_+ +C_1 C_- \right),
\end{align*} where the quantities $C_1,C_2,C_{-},C_+$ are as in the definition \ref{defn:C}.
\end{lemma}
\begin{proof}  We use the previous lemma iteratively to get
\begin{align} \label{eq: leftmom}
\int_{v<0}|v|f(0,v)\mathrm{d}v =& \int_{v<0} |v| \left( e^{-1/\kappa |v|}f(1,v) + \int_0^1 \frac{1}{\kappa|v|} e^{-y/\kappa |v|} \rho(y) \mathcal{M}_{T(y)}(v) \mathrm{d}y \right)\mathrm{d}v \\
=& \int_{v<0}|v| e^{-1/\kappa |v|} f(1,v) \mathrm{d}v + C_-  \notag \\
=& \int_{v<0} |v| e^{-1/\kappa |v|} \widetilde{\mathcal{M}}_2(v) \mathrm{d}v \int_{v'>0} |v'| f(1,v') \mathrm{d}v' + C_- \notag \\
=& C_2 \int_{v>0}|v| f(1,v) \mathrm{d}v +C_-. \notag
\end{align}
Similarly
\begin{align} \label{eq:rightmom }
\int_{v>0} |v| f(1,v) \mathrm{d}v =& \int_{v>0} \left( |v| e^{-1/\kappa |v|}f(0,v) + \frac{1}{\kappa}\int_0^1 e^{-(1-y)/\kappa |v|} \rho(y) \mathcal{M}_{T(y)}(v) \mathrm{d}y \right) \mathrm{d}v \\
=& \int_{v>0} |v| e^{-1/\kappa |v|} f(0,v) \mathrm{d}v + C_+ \notag \\
=& \int_{v>0} |v| e^{-1/\kappa |v|} \widetilde{\mathcal{M}}_1(v) \mathrm{d}v \int_{v'<0} |v'| f(0,v') \mathrm{d}v' + C_+ \notag \\
=& C_1 \int_{v<0} |v| f(0,v) \mathrm{d}v + C_+. \notag
\end{align} Substituting \eqref{eq:rightmom } into \eqref{eq: leftmom} gives  the result.
\end{proof}

Combining the previous two Lemmas we get easily the following representation of the solution as stated in the following Lemma. 
\begin{lemma} \label{lem: representation of sol}
\begin{align}
f(x,v) =& e^{-x/\kappa |v|}\widetilde{\mathcal{M}}_1(v) \frac{C_- + C_2 C_+}{1-C_1C_2} + \int_0^x \frac{1}{\kappa|v|}e^{-(x-y)/\kappa|v|}\rho(y) \mathcal{M}_{T(y)}(v) \mathrm{d}y \qquad v>0, \label{repres of f for v>0} \\
f(x,v) =& e^{-(1-x)/\kappa|v|} \widetilde{\mathcal{M}}_2(v) \frac{C_+ + C_1 C_-}{1-C_1 C_2} + \int_x^1 \frac{1}{\kappa |v|}e^{-(y-x)/\kappa |v|} \rho(y) \mathcal{M}_{T(y)}(v) \mathrm{d}y \qquad v<0. \label{repres of f for v<0}
\end{align}
\end{lemma}

We remind here that we are aiming for estimates on $\mathscr{F}(T)$ which is defined as following:  if $f$ is a solution to \eqref{eq:steady}-\eqref{bc:left}-\eqref{bc:right} with profile $T$ then
\[ \mathscr{F}(T)(x) = \frac{\int f(x,v)|v|^2 \mathrm{d}v}{\int f(x,v) \mathrm{d}v}. \] In particular, 
using the following definitions for the \textit{hydrodynamic moments} 
\begin{align*}
\rho_T(x) =& \int f(x,v) \mathrm{d}v,\\
\rho_T(x)u_T(x) =& \int f(x,v) v \mathrm{d}v \\
\rho_T(x)(\tau_T(x) + u_T(x)^2) =& P_T(x) = \int f(x,v) v^2 \mathrm{d}v,
\end{align*}
we would like to show that if $T(x) \in [T_1, T_2]$ then $\tau_T(x) \in [T_1,T_2]$.\\

Therefore we are interested in the scalings of the different quantities $\rho_T, P_T,\tau_T$ in terms of the temperatures $T_1,T_2,T(y) \rightarrow \infty$. These asymptotic behaviours are presented in the following series of Lemmas.  

\begin{lemma} \label{lem: est C_1,C_2}
As $T_1,T_2 \rightarrow \infty$ we have
\[ \frac{1}{1-C_1C_2} \sim \sqrt{\frac{2}{\pi}}\frac{\kappa\sqrt{T_1T_2}}{\sqrt{T_1}+\sqrt{T_2}}. \]
\end{lemma}
\begin{proof}
Let us write $D_1 = 1-C_1$ and compute that
\[ C_1 = \sqrt{\frac{2\pi}{T_1}}\int_0^\infty v e^{-1/\kappa v} (2\pi T_1)^{-1/2} e^{-|v|^2/2T_1}\mathrm{d}v = \int_0^\infty ue^{-1/(\kappa \sqrt{T_1}u)} e^{-u^2/2}\mathrm{d}u. \] Therefore, 
\[ D_1 = \int_0^\infty u \left(1-e^{-1/(\kappa \sqrt{T_1}u)} \right) e^{-u^2/2} \mathrm{d}u. \]
 We can straightforwardly bound this above to get
\[ D_1 \leq \int_0^\infty \frac{1}{\kappa\sqrt{T_1}} e^{-u^2/2} \mathrm{d}u = \frac{1}{\kappa}\sqrt{\frac{\pi}{2T_1}}. \] 
In order to bound $D_1$ below, first Taylor expanding gives us:
\[ e^{-1/(\kappa \sqrt{T_1}u)} \leq \max \left\{1, 1-\frac{1}{\kappa\sqrt{T_1}u} + \frac{1}{2\kappa^2 T_1u^2} \right\}. \] Therefore,
\[ u \left( 1- e^{-1/(\kappa \sqrt{T_1}u)} \right) \geq \frac{1}{\kappa \sqrt{T_1}} \max \left\{0, 1- \frac{1}{2\kappa \sqrt{T_1}u} \right\}. \]
So that for any $\alpha \in (0,1)$
\begin{align*}
D_1 \geq \frac{1}{\kappa \sqrt{T_1}} \int_0^\infty \max \left\{0, 1- \frac{1}{2 \kappa \sqrt{T_1}u} \right\} e^{-u^2/2} \mathrm{d} u =& \frac{1}{\kappa \sqrt{T_1}} \int_{\frac{1}{2\kappa \sqrt{T_1}}}^\infty \left(1-\frac{1}{2\kappa \sqrt{T_1}u}\right) e^{-u^2/2} \mathrm{d}u\\
\geq & \frac{1}{\kappa \sqrt{T_1}} \int_{1/(2\kappa \sqrt{T_1}\alpha)}^\infty \left(1-\frac{1}{2\kappa \sqrt{T_1} u}\right) e^{-u^2/2} \\
\geq &\frac{1}{\kappa \sqrt{T_1}} (1-\alpha) \left(2 \sqrt{\pi} - \frac{1}{2\alpha \kappa \sqrt{T_1}} \right) \\
=&  2 \frac{1}{\kappa}\sqrt{\frac{\pi}{T_1}} - \frac{1}{2\alpha \kappa^2 T_1} - 2\alpha \frac{1}{\kappa}\sqrt{\frac{\pi}{T_1}} + \frac{1}{2\kappa^2 T_1}.
\end{align*} Optimising over $\alpha$, for $\kappa^2 T_1> 1/2\pi$,  gives
\[ D_1 \geq 2\frac{1}{\kappa}\sqrt{\frac{\pi}{T_1}} + 
\frac{1}{2\kappa^2 T_1} - 
2\left(\frac{\pi}{\kappa^6 T_1^3} \right)^{1/4}. \]
Symmetrically we find that for $T_2> 1/2\pi$, 
\[ 2\frac{1}{\kappa}\sqrt{\frac{\pi}{T_2}} + 
\frac{1}{2\kappa^2T_2} - 
2\left(\frac{\pi}{\kappa^6 T_2^3}\right)^{1/4} \leq D_2 \leq \sqrt{\frac{\pi}{2\kappa^2 T_2}}.  \]
We can rewrite
\[ 1-C_1C_2 = D_1 + D_2 - D_1D_2. \] 
Lets write 
\[ E_i := - \frac{1}{\kappa}\frac{1}{\sqrt{2\pi T_i}} + \left(\frac{8}{\pi\kappa^2 T_i} \right)^{1/4} \sim A\kappa^{-1/2}T_i^{-1/4} \]Therefore our upper and lower bounds give
\[ 1-C_1C_2 \leq \sqrt{\frac{\pi}{2}} \left( \frac{1}{\kappa \sqrt{T_1}} + \frac{1}{\kappa \sqrt{T_2}} - \sqrt{\frac{\pi}{2\kappa^4 T_1T_2}} (1-E_1)(1-E_2)\right). \] and
\[ 1-C_1C_2 \geq  \sqrt{\frac{\pi}{2}} \left(\frac{1}{\kappa\sqrt{T_1}}(1-E_1) + \frac{1}{\kappa\sqrt{T_2}}(1-E_2) - \sqrt{\frac{\pi}{2\kappa^4T_1T_2}} \right). \]
Therefore we have 
\[ 1- K_1 \leq \frac{1-C_1C_2}{\sqrt{\frac{\pi}{2}}\left( \frac{1}{\kappa\sqrt{T_1}} + \frac{1}{\kappa\sqrt{T_2}}\right)} \leq 1, \] where
\[ K_1 = \frac{\kappa\sqrt{T_2}E_1 + \kappa\sqrt{T_1}E_2 +1}{\kappa\sqrt{T_1}+\kappa\sqrt{T_2}} \leq E_1 + E_2 + \frac{1}{\kappa(\sqrt{T_1} + \sqrt{T_2})}.  \]
So we end up with
\[ \sqrt{\frac{2}{\pi}} \frac{\kappa\sqrt{T_1 T_2}}{\sqrt{T_1}+\sqrt{T_2}} \leq \frac{1}{1-C_1C_2} \leq \sqrt{\frac{2}{\pi}} \frac{\kappa\sqrt{T_1 T_2}}{\sqrt{T_1}+\sqrt{T_2}} \frac{1}{1 - E_1 -E_2 -1/(\kappa(\sqrt{T_1}+\sqrt{T_2}))}. \] We can also straightforwardly check that if $\kappa^2 T_1 \geq \gamma_2$, for some constant $\gamma_2>0$,  then we can use the approximation $1/(1-z) \leq 1+z$ to get that
\[  \frac{1}{1 - E_1 -E_2 -1/\kappa(\sqrt{T_1}+\sqrt{T_2})} \leq 1 + 2 \left(\frac{8}{\pi \kappa^2 T_1} \right)^{1/4}+2 \left(\frac{8}{\pi\kappa^2 T_2} \right)^{1/4}. \]
\end{proof}

\begin{lemma}
If $T_1 \leq T(y)$ then we have that 
\[ \frac{1}{\kappa} \left(1-2\left(\frac{2}{\pi \kappa^2 T_1} \right)^{1/4}\right) \leq 2C_-, 2C_+ \leq \frac{1}{\kappa}. \]
\end{lemma}
\begin{proof}
We just show this for $C_-$, the proof for $C_+$ is almost identical.
\[ \kappa C_{-} = \int_0^1 \rho(y) \int_{v<0} e^{-y/\kappa|v|}\mathcal{M}_{T(y)}(v) \mathrm{d}v \mathrm{d}y. \] The bound $e^{-y/\kappa|v|} \leq 1$ gives us the upper bound immediately. 

For the lower bound we look at $D(y) =1-C_-$, and wish to bound this above.
\[ D = \int_0^1 \rho(y) \int_{v<0} (1-e^{-y/\kappa|v|}) \mathcal{M}_{T(y)}(v) \mathrm{d}v \mathrm{d}y.\] We look at the integral first in $v$ and change variables
\[ \int_{v<0} (1-e^{-y/\kappa|v|}) \mathcal{M}_{T(y)}(v) \mathrm{d}v  = \int_0^\infty \left(1-e^{-y/(\kappa\sqrt{T(y)}v)}\right) \mathcal{M}_1(v) \mathrm{d}v. \] For any $\alpha \in (0,1)$ we use the bounds
\[ 1-e^{-y/(\kappa\sqrt{T(y)}v)} \leq 1, \hspace{5pt} |v| \leq \frac{y}{\alpha \kappa\sqrt{T(y)}}, \] and
\[ 1-e^{-y/(\kappa\sqrt{T(y)}v)} \leq \alpha, \hspace{5pt} |v| > \frac{y}{\alpha \kappa\sqrt{T(y)}}. \] This gives us
\[  \int_{v<0} (1-e^{-y/\kappa|v|}) \mathcal{M}_{T(y)}(v) \mathrm{d}v  \leq \frac{y}{\alpha \sqrt{2\pi \kappa T(y)}} + \frac{\alpha}{2}. \] We optimise over $\alpha$ to get
\[  \int_{v<0} (1-e^{-y/\kappa|v|}) \mathcal{M}_{T(y)}(v) \mathrm{d}v  \leq 2 \left( \frac{2y^2}{\pi \kappa^2T(y)} \right)^{1/4} \leq 2 \left(\frac{2}{\pi\kappa^2 T_1} \right)^{1/4}. \] We then use the fact that $\rho$ integrates to one to conclude.
\end{proof}

Combining the above two lemmas gives us the scaling of the quantity appearing in the first term of the  representation \eqref{repres of f for v>0}. Note that similar calculations will give same scaling regarding \eqref{repres of f for v<0}.

\begin{lemma} \label{lem: scal of term from bc} We have 
\begin{align} 
F_1(T_1,T_2) & \sqrt{\frac{2}{\pi}}\frac{\sqrt{T_1 T_2}}{\sqrt{T_1}+\sqrt{T_2}} \leq \frac{C_-+C_2 C_+}{1-C_1C_2} 
\leq  F_2(T_1,T_2)\sqrt{\frac{2}{\pi}} \frac{\sqrt{T_1 T_2}}{\sqrt{T_1}+\sqrt{T_2}} \notag
\end{align} where
 
\begin{align} \label{definition F_1}
F_1(T_1,T_2):= 1+ \left(\frac{\pi}{2\kappa^6T_1^3}\right)^{1/4} - 2 \left(\frac{2}{\pi \kappa^2 T_1} \right)^{1/4} -  \frac{1}{2} \sqrt{ \frac{\pi}{2\kappa^2 T_1}}  
\end{align}
and
\begin{align}  \label{definit F_2}
F_2(T_1,T_2) :=1+ \frac{1}{4\kappa^2 T_2}- \sqrt{\frac{\pi}{\kappa^2 T_2}} - \left(\frac{\pi}{\kappa^6 T_2^3}\right)^{1/4}. 
\end{align}
\end{lemma}
\begin{proof}
We just put together the previous two Lemmas.
\end{proof}

We would now like to get a sense of the different quantities using these results. Let us start with the pressure $P_T$.  
\begin{lemma}  \label{lemm: bounds on pressure}
We have for all $x \in (0,1)$,
\begin{align}
 G_1(T_1,T_2) \sqrt{T_1T_2} \leq P_T(x) \leq  G_2(T_1,T_2) \sqrt{T_1T_2},
\end{align} 
where 
$$ G_1(T_1,T_2) = F_1(T_1,T_2)\left(1- \sqrt{\frac{2}{\pi}} \frac{1}{\kappa(\sqrt{T_1}+\sqrt{T_2})} \right), $$ 
$$ G_2(T_1,T_2) = F_2(T_1,T_2) + \sqrt{\frac{1}{2\pi \kappa^2T_1}} $$
and $F_1,F_2$ are functions of the temperatures $T_1,T_2$ and they are defined in Lemma \ref{lem: scal of term from bc}. 
In particular, $$  \sqrt{T_1T_2}\lesssim P_T(x) \lesssim  \sqrt{T_1T_2}.$$
\end{lemma}
\begin{proof} Let us first note that the pressure is constant in $x$. Indeed, from the equation we can easilty see that $$\partial_x(\rho(x) u_T(x)) = 0$$ and from the boundary conditions we have $$\rho(0)u_T(0) = \rho(1)u_T(1) =0,$$ hence $u_T(x) =0$. Since we have that $\partial_x P_T(x) = - \frac{1}{\kappa}\rho(x) u(x) =0$, we know that $P_T(x)$ is constant.\\

Now in order to quantify it in terms of the temperatures,  we need two further quantities:
\[ \int_{v>0} |v|^2 \widetilde{\mathcal{M}}_i(v) \mathrm{d}v =\sqrt{\frac{\pi T_i}{2}}, \] which is straightforward to compute. We also show that
\[ \sqrt{\frac{T_1}{2\pi}} - \frac{1}{2\kappa} \leq \int_0^1 \rho(y) \int_0^\infty |v|e^{-y/\kappa |v|} \mathcal{M}_{T(y)}(v) \mathrm{d}v \mathrm{d}y \leq \sqrt{\frac{T_2}{2\pi}}. \] The upper bound comes from bounding $e^{-y/\kappa |v|}$ by one, the lower bound comes from bounding it below by $1-y/\kappa |v|$. 

More precisely, for $v$ positive, using the representation of the solution in \eqref{repres of f for v>0} for the upper bound we write
$$  \int_0^\infty |v|^2 f(x,v)dv  =  \int_0^{\infty} v^2 e^{-x/\kappa v} \widetilde{\mathcal{M}}_1(v) dv \left( \frac{C_{-}+C_2C_+}{1-C_1C_2}\right) + \frac{1}{\kappa}\int_0^{\infty} \int_0^x \rho(y)|v| e^{-(x-y)/\kappa v}\mathcal{M}_{T(y)}(v) dydv$$ and since the pressure is constant, for $x=1$, the above quantity is bounded above by $$\int_0^{\infty} |v|^2 f(1,v)dv \le \sqrt{T_1} \frac{\sqrt{T_1T_2}}{\sqrt{T_1}+\sqrt{T_2}} F_2(T_1,T_2)+  \sqrt{\frac{T_2}{2\pi}}. $$ while the lower bound similarly is found to be  $$ \left( \sqrt{\frac{\pi T_1}{2}}-1 \right)\sqrt{\frac{2}{\pi}} \frac{\kappa \sqrt{T_1T_2}}{\sqrt{T_1}+\sqrt{T_2}} F_1(T_1,T_2)  $$
\\
For $v$ negative, using the representation \eqref{repres of f for v<0}, we have
$$ \int_{-\infty}^0|v|^2f(1,v)\mathrm{d}v =  \int_{-\infty}^0|v|^2 \widetilde{\mathcal{M}}_2(v) \frac{C_++C_1C_-}{1-C_1C_2}\mathrm{d}v. $$ Therefore we can bound it above by,
$$ \sqrt{T_2} \frac{\sqrt{T_1T_2}}{\sqrt{T_1}+\sqrt{T_2}} F_2(T_1,T_2) $$ and below by,
$$ \sqrt{T_2} \frac{\sqrt{T_1T_2}}{\sqrt{T_1}+\sqrt{T_2}} F_1(T_1,T_2). $$
Summing over positive and negative velocities gives us that
$$ P_T(x) = P_T(1) \leq  \sqrt{T_1T_2} \left(F_2(T_1,T_2) + \sqrt{\frac{1}{2\pi \kappa^2 T_1}} \right).  $$ Similarly, we get the lower bound,
$$ P_T(x)= P_T(1) \geq  \sqrt{T_1T_2} F_1(T_1,T_2)\left(1- \sqrt{\frac{2}{\pi}} \frac{1}{\kappa(\sqrt{T_1} + \sqrt{T_2})} \right) $$

\end{proof}

The following Lemma concerns the asymptotics of the density $\rho$.

\begin{lemma} \label{lem: bounds on rho}
We have, uniformly in $x$,
\[ 1- \gamma_0 \kappa^{-1/2}T_1^{-1/4} \leq \rho_T(x) \leq 1+ \gamma_1 \kappa^{-1/2}T_1^{-1/4}. \]
for some constants $\gamma_0, \gamma_1$. 
\end{lemma}
\begin{proof}
 Looking at the formulae \eqref{repres of f for v>0} and \eqref{repres of f for v<0},  we have
 \begin{align*}
     \int_0^\infty f(x,v) \mathrm{d}v &= \int_{v>0} e^{-x/\kappa v} \widetilde{\mathcal{M}}_1(v) \frac{C_{-}+C_2C_{+}}{1-C_1C_2} + \int_0^x \rho(y) \int_0^\infty \frac{1}{\kappa v} e^{-(x-y)/\kappa v}\mathcal{M}_{T(y)}(v)\mathrm{d}v \mathrm{d}y \\ &\le
F_2(T_1,T_2) \frac{ \sqrt{T_1T_2}}{\sqrt{T_1}+\sqrt{T_2}} \int_{v>0} \widetilde{\mathcal{M}}_1(v) \mathrm{d}v + \int_0^x \rho(y) \int_0^\infty \frac{1}{\kappa v} e^{-(x-y)/\kappa v}\mathcal{M}_{T(y)}(v)\mathrm{d}v \mathrm{d}y \\ &:= I_1 + I_2
 \end{align*} where we remind that $F_2$ is given by \eqref{definit F_2}. 
 For $I_1$ applying  the above estimates we get $$ I_1 \le F_2(T_1,T_2)\sqrt{\frac{\pi}{2}} \frac{ \sqrt{T_2}}{\sqrt{T_1}+\sqrt{T_2}}. $$
 For the second term $I_2$: first we notice that 
\[ \frac{1}{ z} e^{-(x-y)/z} \leq  \max \left\{ \frac{ e^{-1}}{x-y}, \frac{1}{ z} \right\}. \] Therefore
\begin{align*}
 \int_0^\infty \frac{1}{\kappa v} e^{-(x-y)/\kappa v} \mathcal{M}_{T(y)}(v) \mathrm{d}v =& \int_0^\infty \frac{1}{\kappa \sqrt{T(y)}v} e^{-(x-y)/\kappa\sqrt{T(y)} v} \mathcal{M}_1(v) \mathrm{d}v \\
\leq & \int_0^a e^{-1}\frac{1}{x-y} \mathcal{M}_1(v) \mathrm{d}v + \int_a^ \infty \frac{1}{a\kappa \sqrt{T(y)}} \mathcal{M}_1(v) \mathrm{d}v \\
\leq & \frac{a }{e^1(x-y)\sqrt{2\pi}} + \frac{1}{2a\kappa\sqrt{T(y)}}.
\end{align*} Optimising over $a$ gives that 
\[\int_0^\infty \frac{1}{\kappa v} e^{-(x-y)/\kappa v} \mathcal{M}_{T(y)}(v) \mathrm{d}v  \leq e^{-1/2} \left( \frac{2}{\pi \kappa^2 T(y)} \right)^{1/4}\sqrt{\frac{1}{x-y}}.  \]

Therefore,
\[ I_2 = \int_0^x  \rho(y) \int_0^\infty \frac{1}{v} e^{-(x-y)/v}\mathcal{M}_{T(y)}(v)\mathrm{d}v \mathrm{d}y \leq 2 e^{-1/2} \left( \frac{2}{\pi \kappa^2 T_1} \right)^{1/4}\|\rho\|_\infty  \]

We can do the same thing for negative $v$ and put it together to get that
\begin{align*}  \|\rho\|_\infty &\leq 2e^{-1/2}\left(\frac{2}{\pi\kappa^2 T_1} \right)^{1/4}\|\rho\|_\infty + F_2 .
\end{align*} Rearranging gives
\[ \|\rho\|_\infty \left( 1 - \gamma_1 \kappa^{-1/2} T_1^{-1/4}\right) \leq  F_2 . \] Hence,
\[ \|\rho \|_\infty \leq   F_2(1-\gamma_1 \kappa^{-1/2}T_1^{-1/4})^{-1}. \]

For a lower bound on $\rho$ we can completely ignore the term where we integrate in $y$ in the formula \eqref{repres of f for v>0}. So we just need to bound below terms like
\[ \int_0^\infty e^{-1/\kappa v} \widetilde{\mathcal{M}}_1(v) \mathrm{d}v. \] We have already treated terms of this type we bound the integrand below by $1-\alpha$ for $v \geq 1/\alpha\sqrt{T_1}$ and optimise over $\alpha$ to get
\[  \int_0^\infty e^{-1/\kappa v} \widetilde{\mathcal{M}}_1(v) \mathrm{d}v \geq \sqrt{\frac{\pi}{2 \kappa^2 T_1}} \left( 1 - 2\left( \frac{2}{\pi\kappa^2 T_1} \right)^{1/4} \right).  \] 
Therefore, we write 
\begin{align*}
 \int_{v>0} f(x,v) dv  &\gtrsim \sqrt{\frac{\pi}{2 \kappa^2 T_1}} \left( 1 - 2\left( \frac{2}{\pi\kappa^2 T_1} \right)^{1/4} \right) \frac{\sqrt{T_1T_2}}{\sqrt{T_1}+ \sqrt{T_2}} F_1(T_1,T_2)  \end{align*}
Summing over positive and negative $v$ gives
\[ \rho(x)\geq 1- \gamma_0\kappa^{-1/2} T_1^{-1/4}.\]

\end{proof}

Now given the scalings in terms of the temperatures for $P_T(x)$ and $\rho_T(x)$ and the fact that $P_T(x)= \rho_T(x)\tau_T(x)$ for every $x$ we have that
\begin{lemma} \label{lem: bounds on tau}
We have that for all $x \in (0,1)$, asymptotically with $T_1$, 
\[ \sqrt{T_1T_2}(1-  \gamma_1\kappa^{-1/2}T_1^{-1/4}) \lesssim \tau_T(x) \lesssim  \sqrt{T_1T_2}(1 +\gamma_0 \kappa^{-1/2}T_1^{-1/4}). \] 
\end{lemma}
\begin{proof}
This is simply a matter of piecing together the previous lemmas.
\end{proof}


\begin{proof}[Proof of Proposition \ref{prop:linftybounds}]
This follows immediately from the previous lemma, since from the the second item of condition \ref{condition}, \hyperref[con2]{(C2)}, we indeed have that  $$ \sqrt{T_1T_2}(1 +\gamma_0 \kappa^{-1/2}T_1^{-1/4}) \le T_2 $$ and $$  \sqrt{T_1T_2}(1-  \gamma_1 \kappa^{-1/2}T_1^{-1/4}) \ge T_1. $$
\end{proof}

\subsection{H\"older continuity of $\mathscr{F}(T)(x)$.} \label{subsection: Holder continuity}

In this Section we show H\"older continuity of order $1/2$ for the map $\mathscr{F}(T)= \tau$. This will allow us to use Schauder fixed point theorem, see Theorem \ref{Theorem: Schauder},  to get the desired fixed point for $\mathcal{F}$. Again in this section the precise constants do not matter for the final result so we work with $\kappa =1$.
\begin{prop} \label{prop: Holder cont for tau}
If $T(x) \in [T_1,T_2]$ then there exists a constant $C(T_1,T_2)$ such that
\[ |\tau(x_1) - \tau(x_2)| \leq C(T_1,T_2)\sqrt{|x_1-x_2|}. \]
\end{prop}
\begin{proof} For $x_1 \le x_2$ in $(0,1)$ we write 
\begin{align} \label{eq: Holder tau1}  |\tau_T(x_1) - \tau_T(x_2)| &= \left\vert \frac{P}{\rho_T(x_1)}  -   \frac{P}{\rho_T(x_2)} \right\vert   = |P| \left\vert  \frac{\rho_T(x_2) - \rho_T(x_1)}{\rho_T(x_1)\rho_T(x_2)} \right\vert \\ & \le C(T_1,T_2) |\rho_T(x_2) - \rho_T(x_1)|  \notag
\end{align}
where $P = P_T(x)$ is the constant pressure we got from Lemma \ref{lemm: bounds on pressure} and $C(T_1,T_2)$ a constant that depends only on the two temperatures and comes from the upper bound on $P$, Lemma \ref{lemm: bounds on pressure},  and the, uniform in $x$, lower bound on the density as well, see Lemma \ref{lem: bounds on rho}. \\
Thus, in order to conlude we need to prove H\"older continuity for $\rho(x)$. We need to estimate: 

\[ \int_0^\infty (f(x_2,v))-f(x_1,v)) \mathrm{d}v. \] We can split this into two terms
\[ I_1 = \frac{C_- + C_2 C_+}{1-C_1C_2} \sqrt{\frac{2\pi}{T_1}} \int_0^\infty e^{-x_2/v}\left(1-e^{-(x_1-x_2)/v}\right) \mathcal{M}_{T_1}(v) \mathrm{d}v, \] and
\[ I_2 = \int_0^{x_2} \frac{1}{|v|}e^{-(x_2-y)/|v|}\rho(y) \mathcal{M}_{T(y)}(v) \mathrm{d}y - \int_0^{x_1} \frac{1}{|v|}e^{-(x_1-y)/|v|}\rho(y) \mathcal{M}_{T(y)}(v) \mathrm{d}y. \] 
In order to bound $I_1$ from above, we use Lemma \ref{lem: scal of term from bc} to write
\[ I_1 \leq \theta(T_1,T_2) \int_0^\infty \left(1-e^{-(x_2-x_1)/v}\right) \mathcal{M}_{T_1}(v)\mathrm{d}v  \] for some constant $ \theta(T_1,T_2)$. 
Now proceding as in the proofs of the Lemmas in the previous subsection, we split into $v \leq (x_2-x_1)/\alpha$ and bound the integrand by $1$ for small $v$ and by $\alpha$ for large $v$: 
\[ I_1 \lesssim \int_0^{(x_2-x_1)/\alpha} \mathcal{M}_{T_1}(v) \mathrm{d}v + \alpha \int_{(x_2-x_1)/\alpha}^\infty \mathcal{M}_{T_1}(v) \mathrm{d}v \lesssim \frac{x_1-x_1}{\alpha \sqrt{2\pi T_1}} + \frac{\alpha}{2}. \] 
Optimising over $\alpha$ gives
\[ I_1 \leq \theta'(T_1,T_2) \sqrt{(x_2-x_1)}. \] Here the constant $\theta'(T_1,T_2)$ depends only on $T_1,T_2$.
Now we turn to $I_2$. We can rewrite it as 
\[ I_2 = I_3 + I_4, \] with 
\[ I_3 = \int_0^{x_1}\int_{v>0} \frac{1}{v} \rho(y) \mathcal{M}_{T(y)}(v) \left( e^{-(x_2-y)/v}-e^{-(x_1-y)/v} \right) \mathrm{d}v \mathrm{d}y, \]
and
\[ I_4 = \int_{x_1}^{x_2} \int_{v>0} \frac{1}{v} \rho(y) \mathcal{M}_{T(y)}(v) e^{-(x_2-y)/v} \mathrm{d}v \mathrm{d}y. \] Looking first at $I_4$ we show that
\[ I_4 \leq \|\rho\|_\infty \sqrt{\frac{T_2}{T_1}} \int_{x_1}^{x_2} \frac{1}{v} e^{-(x_2-y)/v} \mathcal{M}_{T_2}(v) \mathrm{d}v \mathrm{d}y. \] Integrating in $y$ gives
\[ I_4 \leq Const. \int_{v>0} \left( 1- e^{-(x_2-x_1)/v} \right) \mathcal{M}_{T_2}(v) \mathrm{d}v \leq Const. \sqrt{x_2-x_1}. \]
Now,
\[ I_3 \leq \|\rho\|_\infty \sqrt{\frac{T_2}{T_1}} \int_0^{x_1} \int_{v>0} \frac{1}{v} \mathcal{M}_{T_2}(v) e^{-(x_1-y)/v} \left( 1- e^{-(x_2-x_1)/v}\right) \mathrm{d}v \mathrm{d}y. \] Integrating this in $y$ gives
\[ I_3 \leq \|\rho\|_\infty \sqrt{\frac{T_2}{T_1}} \int_{v>0} \mathcal{M}_{T_2}(v) \left(1-e^{-x_1/v}\right) \left( 1- e^{-(x_2-x_1)/v} \right) \mathrm{d}v. \] We can bound this by
\[ I_3 \leq Const. \int_{v>0} \mathcal{M}_{T_2}(v) \left( 1- e^{-(x_2-x_1)/v}\right) \mathrm{d}v \leq Const. \sqrt{x_2-x_1}. \] 
So we can repeat this for $v<0$ to get
\[ |\rho(x_2) - \rho(x_1)| \leq C(T_1, T_2) \sqrt{x_2-x_1}. \] 
This gives uniform H\"older continuity for $\rho$. Now we get H\"{o}lder continuity for the $\tau(x)$ by combining this with \eqref{eq: Holder tau1}: 
$$ |\tau_T(x_1) - \tau_T(x_2)|  \le  C(T_1, T_2) \sqrt{x_2-x_1}.$$
\end{proof}


\subsection{Continuity of the map $\mathscr{F}$} \label{subsection: contin of F}
 In this subsection we are going to prove the continuity of the map $\mathcal{F}$ which is the second main ingredient in order to apply Schauder's fixed point Theorem.  Here we continue to work with $\kappa=1$.
\begin{prop}\label{prop:ctyF}
The map $\mathscr{F}$ is continuous from $C((0,1))$ to $C((0,1))$ with the $L^\infty$ norm.
\end{prop}
\begin{proof}
Let $T(x), \widetilde{T}(x)$ be two different continuous functions satisfying the bounds, \textit{i.e.} are bounded below and above by $T_1,T_2$ respectively where $T_1,T_2$ satisfy condition \eqref{condition}. 
In order to conclude the continuity of $\mathcal{F}$, we want to estimate the quantity $ | \mathcal{F}(T)(x)- \mathcal{F}(\widetilde{T})(x)|$ and bound it in terms of the difference of the two temperatures $ |T(x) - \widetilde{T}(x)|$ for all $x \in (0,1)$. In what follows we write  $\widetilde{P}:= P_{\widetilde{T}}, \tilde{\rho} := \rho_{\widetilde{T}} $ and we have  \begin{align*} |  \mathcal{F}(T)(x)- \mathcal{F}(\widetilde{T})(x) | &=  \left\vert \frac{P}{\rho}(x) - \frac{\widetilde{P}}{\tilde{\rho}}(x)  \right\vert \\ &\le \frac{|P- \widetilde{P}|}{\rho}(x)+ \frac{P\widetilde{P} |\rho - \tilde{\rho}|}{\rho \tilde{\rho}}(x).
\end{align*}

Therefore we need to estimate the differences between the two densities and the two pressures that correspond to the two different temperatures. We proceed as in the proofs of the Lemmas in subsection \ref{subsection Linfty bounds}, using Lemma \ref{lem: representation of sol}. We recall the result of Lemma \ref{lem: representation of sol}
\begin{align*}
f(x,v) =& e^{-x/|v|}\widetilde{\mathcal{M}}_1(v) \frac{C_-+ C_2 C_+}{1-C_1C_2} + \int_0^x \frac{1}{|v|}e^{-(x-y)/|v|}\rho(y) \mathcal{M}_{T(y)}(v) \mathrm{d}y \qquad v>0, \\
f(x,v) =& e^{-(1-x)/|v|} \widetilde{\mathcal{M}}_2(v) \frac{C_+ + C_1 C_-}{1-C_1 C_2} + \int_x^1 \frac{1}{|v|}e^{-(y-x)/|v|} \rho(y) \mathcal{M}_{T(y)}(v) \mathrm{d}y \qquad v<0.
\end{align*}

First note that the constants $C_1, C_2$  do not depend on whether we use $T$ or $\widetilde{T}$. However $C_-, C_+$ depend on this. In this case we would like to look at the differences between two different realisations. We recall
\[ C_- = \int_0^1 \rho(y) \int_{-\infty}^0 e^{-y/|v|} \mathcal{M}_{T(y)}(v) \mathrm{d}v \mathrm{d}y. \] Here we first look at the integral in $v$, 
\[ \int_{-\infty}^0 e^{-y/|v|} \mathcal{M}_{T(y)}(v) \mathrm{d}v = \int_0^\infty e^{-y/\sqrt{T(y)}v} \mathcal{M}_1(v) \mathrm{d}v:= F\left(y, \sqrt{T(y)}\right). \] 
We calculate
\begin{align*} \frac{\mathrm{d}}{\mathrm{d}t} F(y,t) =& \int_0^\infty \frac{y}{t^2 v} e^{-y/tv} \mathcal{M}_1(v) \mathrm{d}v \\
=& \int_0^\infty \frac{1}{t} \frac{\mathrm{d}}{\mathrm{d}v} \left( e^{-y/tv} \right) v \mathcal{M}_1(v) \mathrm{d}v\\
=& \int_0^\infty \frac{1}{t} e^{-y/tv} (v^2-1) \mathcal{M}_1(v) \mathrm{d}v  \leq \frac{2}{t}. \end{align*} 
Therefore,
\[\left\vert F\left(y, \sqrt{T(y)}\right) - F\left(y, \sqrt{\widetilde{T}(y)}\right) \right\vert  \leq \frac{2}{\sqrt{T_1}}\left\vert \sqrt{T(y)} - \sqrt{\widetilde{T}(y)} \right\vert.  \] This means that 
\[ |C_- - \widetilde{C}_-| \leq  \frac{2}{\sqrt{T_1}}\big\|\sqrt{T} - \sqrt{\widetilde{T}}\big\|_\infty  + \int_0^1 |\rho(y)-\tilde{\rho}(y)| \int_0^\infty \left(1-e^{-y/v\sqrt{\widetilde{T}(y)}} \right) \mathcal{M}_1(v) \mathrm{d}v\mathrm{d}y. \] Then we can use our bounds from earlier to get that
\[   |C_- - \widetilde{C}_-| \leq \frac{2}{\sqrt{T_1}} \big\|\sqrt{T} - \sqrt{\widetilde{T}}\big\|_\infty  + \sqrt{\frac{\pi}{2T_1}} \| \rho -\tilde{\rho}\|_\infty.\] Exactly the same result is true for $C_+$. \\

Now in order to bound the difference of the densities, we write
\[ B_1(x) = \int_0^\infty e^{-x/v} \widetilde{\mathcal{M}}_1(v) \mathrm{d}v \approx \sqrt{\frac{\pi}{2T_1}}, \] and
\[ B_2(x) = \int_{-\infty}^0 e^{-(1-x)/|v|} \widetilde{\mathcal{M}}_2(v) \mathrm{d} v \approx \sqrt{\frac{\pi}{2T_2}}. \]
 These quantities don't depend on $T,\widetilde{T}$. Let us also write
\[ A_1(x) = \int_0^x \int_0^\infty \frac{1}{v} e^{-(x-y)/|v|} \rho(y) \mathcal{M}_{T(y)}(v) \mathrm{d}v, \] and $A_2(x)$ defined symmetrically. Then we have that
\begin{align} \label{eq: rho}
 \rho(x) = \frac{B_1(x)(C_-+C_2C_+)}{1-C_1C_2} + \frac{B_2(x)(C_++C_1C_-)}{1-C_1C_2} +A_1 +A_2. 
\end{align}
Therefore,
\begin{align} \label{eq: rho - tilderho}
|\rho(x) - \tilde{\rho}(x)| \leq \frac{2(B_1(x)+B_2(x))}{1-C_1C_2}\left(|C_- - \tilde{C}_-| + |C_+ - \tilde{C}_+| \right) + |A_1-\tilde{A}_1| + |A_2 - \tilde{A}_2|. 
\end{align}
We know from Lemma \ref{lem: est C_1,C_2} that $$(1-C_1C_2)^{-1} \approx \sqrt{\frac{2}{\pi}}\frac{\sqrt{T_1T_2}}{\sqrt{T_1}+\sqrt{T_2}}.$$  Therefore, 
\[  \frac{2(B_1(x)+B_2(x))}{1-C_1C_2} \leq 2. \] Therefore we can bound the first term in the rhs of \eqref{eq: rho - tilderho} by
\[ 4(2+ \sqrt{\pi/2}) \frac{1}{\sqrt{T_1}} \left( \big\|\sqrt{T} - \sqrt{\widetilde{T}}\big\|_\infty  + \|\rho - \tilde{\rho}\|_\infty \right). \]
Regarding the differences between the $A_i$'s, we look only at $A_1$, since the other case is the same. Let us write
\[ G(y,t) = \int_0^\infty \frac{1}{vt} e^{-(x-y)/vt} \mathcal{M}_1(v) \mathrm{d}v. \] 
By our earlier calculations, for example in the proof of the Lemma \ref{lem: bounds on rho}, we know that
\[ G(y,t) \leq e^{-1/2} \left( \frac{2}{\pi t^2} \right)^{1/4}. \] Then we can differentiate to see
\[ \frac{\mathrm{d}}{\mathrm{d}t}G(y,t) = \int_0^\infty \left( -\frac{1}{vt^2} + \frac{x-y}{v^2 t^3} \right) e^{-(x-y)/vt} \mathcal{M}_1(v) \mathrm{d}v \leq \frac{2}{t} G(y,t) \leq Ct^{-3/2}. \] This last inequality only holds for $t\geq 1$.  Now we have,
\[ A_1(x)- \widetilde{A}_1(x) \leq \int_0^1 \left(\rho(y) \left| G\left(y, \sqrt{T(y)}\right)-G\left(y, \sqrt{\widetilde{T}(y)}\right) \right| + (\rho(y)-\tilde{\rho}(y)) G\left(y, \sqrt{\widetilde{T}(y)}\right) \right)\mathrm{d}y.  \] Therefore,
\[ |A_1(x) - \widetilde{A}_1(x)| \leq CT_1^{-3/4} \big\|\sqrt{T} - \sqrt{\widetilde{T}}\big\|_\infty + CT_1^{-1/4} \|\rho - \tilde{\rho}\|_\infty. \] Therefore overall,
\[ \|\rho - \tilde{\rho}\|_\infty \leq C T_1^{-1/4} \|\rho - \tilde{\rho}\|_\infty + C T_1^{-1/2}\big\|\sqrt{T} - \sqrt{\widetilde{T}}\big\|_\infty. \] Hence,
\begin{align} \label{eq: rho- tilderho}
\|\rho - \tilde{\rho}\|_\infty \leq CT_1^{-1/2} \big\|\sqrt{T} - \sqrt{\widetilde{T}}\big\|_\infty .  
\end{align}

Regarding the estimation of the difference of the pressures we define the following.
\[ D_1(x) = \int_0^\infty |v|^2 e^{-x/v} \widetilde{\mathcal{M}}_1(v) \mathrm{d}v,\ D_2(x) = \int_{-\infty}^0 |v|^2 e^{-(1-x)/v} \widetilde{\mathcal{M}}_2(v) \mathrm{d}v. \] These quantites do not depend on $T$ or $\widetilde{T}$ and we have that
\[ D_1 \approx \sqrt{T_1}, D_2 \approx \sqrt{T_2}. \] Furthermore, we have
\[ E_1(x) = \int_0^x  \int_0^\infty |v| e^{-(x-y)/v} \rho(y) \mathcal{M}_{T(y)}(v) \mathrm{d}v \mathrm{d}y, \] and
\[ E_2(x) = \int_x^1 \int_{-\infty}^0 |v| e^{-(y-x)/|v|} \rho(y) \mathcal{M}_{T(y)}(v) \mathrm{d}v \mathrm{d}y. \]
Then the formula for the pressure can be rewritten as follows
\[ P = \frac{D_1(C_- + C_2 C_+) + D_2(C_+ + C_1 C_-)}{1-C_1C_2} + E_1 + E_2. \] Therefore,
\[ |P -\widetilde{P}| \leq \frac{2(D_1+D_2)}{1-C_1C_2} \left( |C_- -\tilde{C}_-| + |C_+ - \tilde{C}_+| \right) + |E_1 - \tilde{E}_1| + |E_2 - \tilde{E}_2|. \]
So we can bound,
\[ \frac{2(D_1+D_2)}{1-C_1C_2} \leq 2(\sqrt{T_1} + \sqrt{T_2}) \frac{\sqrt{T_1T_2}}{\sqrt{T_1}+\sqrt{T_2}} =2\sqrt{T_1T_2}. \] Then we want to bound the first term by
\[  C\sqrt{T_2} \big\|\sqrt{T} - \sqrt{\widetilde{T}}\big\|_\infty . \]
In general we can see that,
\begin{align} \label{eq: P-tilde P}
|P - \widetilde{P}| \leq C\sqrt{T_2} \big\|\sqrt{T} - \sqrt{\widetilde{T}}\big\|_\infty .  \end{align}
Finally about the difference in temperatures,  we use the results from Lemmas \ref{lem: bounds on rho} and \ref{lemm: bounds on pressure}: that $\rho, \tilde{\rho} \sim 1$, and $P \sim \sqrt{T_1T_2}$. Combining \eqref{eq: rho- tilderho} and  \eqref{eq: P-tilde P} with the calculations in the beginning of this proof, we have for all $x$
\begin{align*} 
 \left\vert  \mathcal{F}(T)(x)-\mathcal{F}(\widetilde{T})(x) \right\vert  &\leq   \frac{C \sqrt{T_2}}{1+\kappa_1T_1^{-1/4}}\big\|\sqrt{T} - \sqrt{\widetilde{T}}\big\|_\infty  + \frac{(T_1T_2)C}{(1+\kappa_1 T_1^{-1/4})^2} \big\|\sqrt{T} - \sqrt{\widetilde{T}}\big\|_\infty  \\ &\le \Lambda(T_1,T_2) \big\|\sqrt{T} - \sqrt{\widetilde{T}}\big\|_\infty 
  \end{align*}  for a constant $\Lambda(T_1,T_2)$ that depends only on the temperatures $T_1,T_2$.
 Since $T, \widetilde{T}$ are bounded below by $T_1$ this gives the required continuity.

\end{proof}

\subsection{Fixed Point Argument} \label{subsection: fixed point thrm}

In this subsection we show how an application of  Schauder's fixed point theorem yields the main result. 
First, for completeness we remind here the Schauder's Theorem, which can be found for example in \cite[Theorem 2.3.7]{Smar74}. 
\begin{thm}[Schauder Fixed Point Theorem] \label{Theorem: Schauder}
Let $S$ be a non-empty, convex closed subsect of a Hausdorff topological vector space and $F$ a mapping of $S$ into itself so that $F(S)$ is compact then $F$ has a fixed point.
\end{thm}
\begin{proof} [Proof of Theorem \ref{thm:main}]  
We apply Schauder's Theorem to get a fixed point for $\mathcal{F}$: \\
Firstly by Proposition \ref{prop:ctyF} we know that the map $\mathscr{F}: C(0,1) \rightarrow C(0,1)$ is a continuous map. For $T_1, T_2$ fixed temperatures satisfying condition \ref{condition}, we have that  $ T(x) \in  [T_1, T_2]$ implies $ \tau \in
[T_1,T_2]$, in other words if we define the set $$ S_{T_1,T_2}:= \{ T \in C([0,1]): T_1 \le T(x) \le T_2  \}, $$ then $$ \mathcal{F}(S_{T_1,T_2}) \subset S_{T_1,T_2}. $$ Also, from Proposition \ref{prop: Holder cont for tau}, we have that $\mathcal{F}(S_{T_1,T_2})$ satisfies a H\"older condition of order $1/2$ with  a constant depending only on the two fixed temperatures. Since moreover $\mathcal{F}(S_{T_1,T_2})$ is uniformly bounded, we conclude by Arzela-Ascoli, the compactness of the set.   

The existence of a fixed point for this mapping ensures us that the steady state for the linear BGK model with temperature profile $T(x)$ is a steady state for the original non-linear model \eqref{eq:nl}-\eqref{BC1}-\eqref{BC2} as well since $T(x)=T_f(x)$.
The properties of this non-equilibrium steady state listed in the statement are proved in Lemmas  \ref{lem: bounds on rho}, \ref{lem: bounds on tau} and Proposition \ref{prop:ctyF} respectively.
\end{proof}

\section{Discussion of the results and future work} \label{sec: discussion}

 First we look at how Fourier's law applies in this specific context. The heat flux associated to the NESS is constant along the interval: 
Recall that the temperature that corresponds to the stationary solution $f$ is given by $$T_f(x) = \frac{1}{\rho_f(x)} \int |v-u(x)|^2f(x,v) dv$$   and the heat flux is the vector field $$ J(x):=\int (v-u(x))|v-u(x)|^2f(x,v) dv.$$ Here
we have $u(x)=0$ everywhere (see beginning of proof of Lemma \ref{lemm: bounds on pressure}). We easily get that $$\partial_x J(x) =0\quad \text{for}\ x \in (0,1),$$ i.e. the heat flux is constant. This means that if Fourier's law \eqref{eq: Fourier law} holds, then $\kappa(T_f)\partial_x T_f(x) $ is constant. Note that this conclusion can be  also found in \cite[proof of Theorem 1.5]{EGKM13} for the full Boltzmann operator and for temperatures close to equilibrium. There, through comparison with numerical simulations  indicating that the temperature is a nonlinear function, one can see that Fourier's law is violated in the kinetic regime. 

 In our seting $T_f(x)$ is close to the constant function $\sqrt{T_1T_2}$ as $T_1\rightarrow \infty$,  as described in the main Theorem \ref{thm:main}. So for large boundary temperatures $T_1$, the function $T_f(x)$ is constant in the bulk of the domain $(0,1)$ which is reminiscent to the behaviour of the heat flux in the harmonic atom networks. 

\smallsection{Comparison with the heat flux in the microscopic harmonic atom chains} In the case of harmonic oscillator chains the temperature profile is similarly close to being constant as shown in \cite{RLL67}, at least in the case of small temperature difference. However in this case the temperature is the centre of the chain is the linear average $(T_1+T_2)/2$ whereas in our case the temperature is close to $\sqrt{T_1T_2}$. In the paper \cite{RLL67} they also show that the temperature is paradoxically lower than the average very close to the hot reservoir. It would be interesting to see if a similar effect could be observed in our model.

As regards the connection of microscopic oscillator chains with the Boltzmann equation for phonons, it has been shown that one can derive a phonon Boltzmann equation as a kinetic limit (high frequency limit) starting from infinite chain of harmonic oscillators with a small anhamonicity. We refer to \cite{Spohn06} for that and to \cite{BOS10} for a stochastically perturbed version of it. Another very interesting work analysing the kinetic limit in case of an infinite linear chain of oscillators coupled to a single Langevin thermostat at the boundary is \cite{KORS20}. 

\smallsection{Possible directions}
The most natural and important question arising from these results is uniqueness of the steady state given here. A less ambitious question in the same direction is whether the steady state found here is stable under small perturbations, this is done in the Boltzmann equation setting in \cite{AEMN10, AEMN12} and in the BGK setting in \cite{CELMM18, CELMM19}. This would also be an interesting question in terms of the study of hypocoercivity as there are only a small number of works showing hypocoercivity for equations on bounded domains and these are generally in the context of the Boltzmann equation intiated by \cite{YGuo10}. Showing hypocoercivity for equations with non-explicit non-equilibrium states is also a significant challenge.

Another possible angle for future work is to investigate similar problems in higher dimensions. This would involve looking at the non-linear BGK equation where $x \in \Omega \subset \mathbb{R}^d$ and $\Omega$ is a smooth bounded domain. In this case getting $L^\infty$ estimates on the solution from the Langrangian expansion of the steady state becomes much more challenging.

\smallsection{Acknowledgements} 
We would like to thank Cl\'ement Mouhot  for many insightful discussions and encouragement and Helge Dietert for sharing his observations for possible connections with the harmonic oscillator chains. JE was supported by FSPM postdoctoral fellowship and the grant ANR-17-CE40-0030,  AM was supported by the EPSRC grant EP/L016516/1 for the University of Cambridge CDT, the CCA.

 \bibliographystyle{alpha}
\bibliography{bibliography}
\end{document}